\documentclass[12pt]{article}
\usepackage[margin=1in]{geometry}
\usepackage[T1]{fontenc}
\usepackage[utf8]{inputenc}
\usepackage{lmodern,amsmath,amssymb,amsthm,microtype}
\usepackage[authoryear,round]{natbib}
\usepackage{setspace}
\usepackage[hidelinks]{hyperref}
\newtheorem{theorem}{Theorem}[section]
\newtheorem{lemma}[theorem]{Lemma}

\newtheorem{corollary}[theorem]{Corollary}
\theoremstyle{definition}
\newtheorem{definition}[theorem]{Definition}
\theoremstyle{remark}
\newtheorem{remark}[theorem]{Remark}
\newcommand{\Pp}{\mathrm{P}}

\newcommand{\R}{\mathbb R}
\title{Extending the $V$-sequence reduction to arbitrary independent distributions}
\author{Jos\'e A. Islas\\
\small Facultad de Ciencias Naturales y Exactas\\
\small Universidad Aut\'onoma de Sinaloa, Culiac\'an, Sinaloa, Mexico\\
\small Corresponding author: \href{mailto:islas.jose@uas.edu.mx}{islas.jose@uas.edu.mx}}
\date{}
\begin{document}
\doublespacing
\maketitle
\begin{abstract}
Let $X_1,\ldots,X_n$ be independent random variables with arbitrary known distributions, where $n\ge2$. The optimal probability $V_n(X_1,\ldots,X_n)$ of choosing an overall maximum satisfies
\[
V_n(X_1,\ldots,X_n)\ge\left(1-\frac1n\right)^{n-1},
\]
and this bound is attained. We obtain this known sharp bound by extending the reduction lemma of Allaart and Islas (2016) from continuous to arbitrary distributions. We construct a $V$-sequence whose optimal win probability does not exceed that of the original sequence, and then apply the bound for $V$-sequences. The construction first replaces atoms by uniform observations on strictly ordered intervals and applies the original continuous reduction lemma. The reduced distributions are determined by a common critical level and a common probability for splitting a possible atom at that level.
\end{abstract}


\section{Introduction and main result}
Let $X_1,\ldots,X_n$ be independent real-valued random variables with known distribution functions $F_1,\ldots,F_n$. Write
\[
M_n:=\max\{X_1,\ldots,X_n\},\qquad
V_n(X_1,\ldots,X_n):=\sup_\tau\Pp(X_\tau=M_n),
\]
where the supremum is over stopping times taking values in $\{1,\ldots,n\}$ and adapted to the natural filtration of $X_1,\ldots,X_n$. In particular, selecting any observation tied for the maximum is a win.

This note continues the reduction method of \citet{AI}. For independent random variables with continuous distributions, that paper constructs a $V$-sequence by replacing the observations one at a time without increasing the optimal win probability. The first observation of the reduced sequence is constant, and each later observation takes at most two values. The purpose here is to extend this construction to arbitrary independent distributions.

Although ties should make it easier to choose an overall maximum, the original proof cannot be applied directly when atoms are present. Continuity is used at two steps. First, it gives a critical value at which the probabilities of winning by stopping and by continuing are equal. With atoms, the two functions may cross through a jump without being equal at any point. Second, an intermediate value argument is used to choose constants needed for the successive replacements. A discontinuous function may skip the required value. Thus the expected extension of the reduction still requires a proof.

We overcome these difficulties by first replacing each atom by a uniform observation on an interval, keeping the intervals for distinct original values strictly ordered. We show that this replacement cannot increase the optimal win probability, and then apply the original continuous reduction lemma. On returning to the original observations, its critical value becomes a common level $c$, with a possible random split of observations equal to $c$. The same splitting probability is used for every observation.

\citet{HK} established sharp bounds for stopping rewards based on ranks for arbitrary independent distributions. Their Theorem 1.2, specialized to selection of the best observation, gives the sharp bound
\[
V_n(X_1,\ldots,X_n)\ge(1-1/n)^{n-1}.
\]
Their convention assigns rank one to every observation tied for the maximum, agreeing with the objective considered here. Their proof also uses auxiliary uniforms to split a threshold atom and transfers the guarantee from randomized to ordinary stopping times. Thus neither this bound nor the use of auxiliary randomization to handle atoms is claimed as new.

Related treatments of ties include \citet{EHLM}, who use auxiliary randomization in threshold algorithms for nonidentically distributed observations, and \citet{GKS}, whose Section 2.2 compares arbitrary i.i.d. observations with continuous ones through randomized atom spreading. Our purpose is to give an explicit general-distribution extension of the Allaart--Islas construction. The comparison uses the same idea of breaking ties at random. The reduction to a $V$-sequence is obtained by applying the lemma in \citet{AI}.

\begin{definition}
An independent sequence $X_1',\ldots,X_n'$ is a \emph{$V$-sequence} if $X_1'\equiv a_1$ and, for $i\ge2$, $X_i'$ takes values in $\{a_i,b_i\}$, where
\[
a_1<a_2<\cdots<a_n,\qquad a_1>b_2>\cdots>b_n.
\]
Degenerate two-valued distributions are allowed.
\end{definition}

\begin{theorem}[Reduction lemma]\label{thm:main}
Let $n\ge2$ and let $X_1,\ldots,X_n$ be independent real-valued random variables. There exist $c\in\R$, $\rho\in[0,1]$ and a $V$-sequence $X_1',\ldots,X_n'$ with probabilities of the upper values
\begin{equation}\label{eq:pi}
p_i:=\Pp(X_i^\prime=a_i)=\Pp(X_i>c)+(1-\rho)\Pp(X_i=c),\qquad 2\le i\le n,
\end{equation}
such that, writing $q_i=1-p_i$ and $r_i=p_i/q_i$,
\begin{equation}\label{eq:constructivevalue}
0<\prod_{i=2}^n q_i
=V_n(X_1',\ldots,X_n')
\le V_n(X_1,\ldots,X_n).
\end{equation}
The sequence can be constructed using independent uniform random variables $\xi_i$ on $(0,1)$, independent of $X_1,\ldots,X_n$, by setting $X_1'\equiv a_1$ and
\begin{equation}\label{eq:coupled}
X_i'=\begin{cases}
a_i,& X_i>c\text{ or }(X_i=c\text{ and }\xi_i>\rho),\\
b_i,&\text{otherwise},
\end{cases}\qquad i\ge2,
\end{equation}
where $a_1<a_2<\cdots<a_n$ and $a_1>b_2>\cdots>b_n$.
Moreover,
\begin{equation}\label{eq:odds}
\sum_{i=2}^n r_i\le1.
\end{equation}
Each optimal win probability is taken over stopping times adapted to the natural filtration of the sequence in question.
\end{theorem}

The probabilities in \eqref{eq:pi} depend on the given distributions. All values above $c$ are assigned to the upper value, all values below $c$ to the lower value, and a value equal to $c$ is assigned to the upper value with probability $1-\rho$. The same $c$ and $\rho$ are used for every observation. The proof does not use the sharp lower bound.

\section{The continuous case}
We retain the notation and definition of \citet{AI}.

The following results are invoked from the original paper, with the variables renamed. We record also the properties of its construction needed for the stronger statement in Theorem~\ref{thm:main}.
\begin{lemma}[Allaart--Islas reduction lemma and its construction, {\citep[Reduction lemma and its proof, Section 3]{AI}}]\label{lem:publishedreduction}
For independent continuously distributed observations $Y_1,\ldots,Y_n$, there is a critical level $x_1^*$ satisfying $U_1(x_1^*)=W_1(x_1^*)>0$, where
\[
U_k(x)=\prod_{j=k+1}^nG_j(x),\qquad
W_k(x)=\sup_{k<\tau\le n}\Pp(Y_\tau\ge\max\{x,Y_{k+1},\ldots,Y_n\}),
\]
Here $1\le k<n$, $G_j$ is the distribution function of $Y_j$, and the stopping times in $W_k$ are adapted to the future observations. The published construction sets $X_1'\equiv x_1^*$ and, for $i\ge2$, assigns the upper value $a_i$ on $\{Y_i>x_1^*\}$ and the lower value $b_i$ otherwise, with $a_i$ and $b_i$ chosen as in the proof of the original reduction lemma. In particular, the upper values satisfy the additional restrictions imposed in that proof; their ordering alone is not used to justify the replacements. Each successive replacement does not increase the optimal value. In particular,
\[
V_n(Y_1,\ldots,Y_n)\ge U_1(x_1^*)=V_n(X_1',\ldots,X_n').
\]
\end{lemma}
For completeness, the equality follows from the proof of the reduction lemma. After replacing the first observation by $x_1^*$, the value is
\[
\max\{U_1(x_1^*),W_1(x_1^*)\}=U_1(x_1^*).
\]
Each later replacement does not increase this value. On the other hand, stopping at the first observation still wins with probability $U_1(x_1^*)$. Thus the final value equals $U_1(x_1^*)$.
\paragraph{Where continuity is used.}
The original proof first chooses $x_1^*$ with $U_1(x_1^*)=W_1(x_1^*)$, using continuity of $U_1$ and $W_1$. It then uses the intermediate value theorem to choose, for $2\le i\le n-1$, a constant $c_i>x_1^*$ such that
\[
\int_{x_1^*}^{\infty}U_i(x)\,dG_i(x)
=U_i(c_i)\,\Pp(Y_i>x_1^*).
\]
These constants are used to choose the upper replacement values $a_i$. For discrete observations, neither the crossing equality nor this choice of $c_i$ is guaranteed by the original argument. The continuous variables constructed in the next section allow both steps to be used exactly as in \citet{AI}.

\begin{lemma}[Bound for $V$-sequences, {\citep[bound for $V$-sequences, Section 3]{AI}}]\label{lem:publishedbound}
For every $V$-sequence $X_1',\ldots,X_n'$,
\[
V_n(X_1',\ldots,X_n')\ge\left(1-\frac1n\right)^{n-1},
\]
and the bound is attained.
\end{lemma}
We use these results in their published form. It remains to replace the atoms by continuous observations without changing strict comparisons or increasing the optimal win probability.

\section{Proof of the reduction lemma}
The observations $X_1,\ldots,X_n$ in this section have arbitrary independent distributions. Auxiliary uniform variables are denoted by $\xi_i$.

\begin{lemma}[Removing atoms]\label{lem:spread}
There are independent continuously distributed random variables $Y_1,\ldots,Y_n$, constructed from independent pairs $(X_i,\xi_i)$, such that
\begin{equation}\label{eq:order}
X_i<X_j\ \Longrightarrow\ Y_i<Y_j
\end{equation}
and $V_n(X_1,\ldots,X_n)\ge V_n(Y_1,\ldots,Y_n)$.
\end{lemma}
\begin{proof}
Let $A=\{a\in\R:\Pp(X_i=a)>0\text{ for some }i\}$. Each distribution has at most countably many atoms, so the finite union $A$ is at most countable. Assign positive weights $w_a$ with $\sum_{a\in A}w_a<\infty$; for example, use $w_{a_j}=2^{-j}$ after enumerating $A$. For an empty $A$, the sum is zero. Define
\[
w(x)=\begin{cases}w_x,&x\in A,\\0,&x\notin A,\end{cases}
\qquad g(x)=x+\sum_{\substack{a\in A\\a<x}}w_a.
\]
Let $\xi_1,\ldots,\xi_n$ be independent uniform random variables on $(0,1)$, independent of $X_1,\ldots,X_n$. Set
\begin{equation}\label{eq:spread}
Y_i=g(X_i)+w(X_i)\xi_i.
\end{equation}
For $x<z$,
\[
g(z)-g(x)-w(x)=z-x+\sum_{\substack{a\in A\\x<a<z}}w_a>0.
\]
Thus the intervals $[g(x),g(x)+w(x)]$ for distinct $x$ are strictly separated, proving \eqref{eq:order}.

The variables $Y_1,\ldots,Y_n$ are independent, since each depends only on the independent pair $(X_i,\xi_i)$. They also have continuous distributions. Indeed, fix $y$. By the preceding inequality, at most one value $x$ can produce $Y_i=y$. If $x\in A$, its replacement is uniform on an interval and therefore has probability zero of equalling $y$. Otherwise, $\Pp(X_i=x)=0$. Hence $\Pp(Y_i=y)=0$.

Now let $\tau$ be a stopping time for $Y_1,\ldots,Y_n$. At each time its decision is a function of the observations so far. Use the same functions after calculating $Y_i$ from $X_i$ and $\xi_i$, and stop at $n$ if no earlier decision is made. These are valid decision functions, since $g$ is increasing and $w$ is zero outside a countable set. Since strict comparisons are preserved,
\[
\{Y_\tau=\max_jY_j\}\subseteq\{X_\tau=M_n\}.
\]

Fix the auxiliary numbers at $u_1,\ldots,u_n$. The resulting rule, denoted by $\tau_u$, is a stopping time for the original observations: its decision at time $i$ uses only $X_1,\ldots,X_i$ and the fixed numbers $u_1,\ldots,u_i$. More explicitly, use the same decision functions as for $\tau$, substituting $g(X_i)+w(X_i)u_i$ for each observation, and stop at $n$ if the rule has not stopped earlier. Therefore,
\[
\Pp(X_{\tau_u}=M_n)\le V_n(X_1,\ldots,X_n).
\]
Averaging over the independent uniform numbers gives
\[
\Pp(Y_\tau=\max_jY_j)
\le\Pp(X_\tau=M_n)
\le V_n(X_1,\ldots,X_n).
\]
Taking the supremum over $\tau$ proves the lemma.
\end{proof}

\begin{proof}[Proof of Theorem~\ref{thm:main}]
Apply Lemma~\ref{lem:spread} and then directly invoke the construction of \citet[Section 3, reduction lemma]{AI}, recorded in Lemma~\ref{lem:publishedreduction}. With $x_1^*$ as its critical value and $a_1=x_1^*$, the resulting $V$-sequence satisfies
\[
V_n(X_1,\ldots,X_n)\ge V_n(Y_1,\ldots,Y_n)
\ge V_n(X_1',\ldots,X_n')=U_1(x_1^*)>0,
\]
and its upper-value events are $\{Y_i>x_1^*\}$.

To express those events in the original coordinates, define
\[
h(x)=g(x)+w(x)=x+\sum_{\substack{a\in A\\a\le x}}w_a,
\qquad c=\inf\{x:h(x)\ge x_1^*\}.
\]
Because $\sum_{a\in A}w_a$ is finite, the sum of the weights in $(c,c+\varepsilon]$ tends to zero as $\varepsilon\downarrow0$. Thus $h$ is right-continuous. Similarly, the sum of the weights in $[c-\varepsilon,c)$ tends to zero, so the left-hand limit of $h$ at $c$ is $g(c)$. Also $x\le h(x)\le x+\sum_{a\in A}w_a$, so $c$ is finite. By the definition of $c$,
\[
g(c)=h(c-)\le x_1^*\le h(c)=g(c)+w(c).
\]
Strict separation of the intervals in Lemma~\ref{lem:spread} gives uniqueness of $c$. If $w(c)>0$, put $\rho=(x_1^*-g(c))/w(c)$; otherwise put $\rho=1$. In the latter case, $x_1^*=g(c)=h(c)$, and observations equal to $c$ do not exceed the transformed threshold. With this convention, the following event identity holds exactly:
\[
\{Y_i>x_1^*\}=\{X_i>c\}\cup\bigl(\{X_i=c\}\cap\{\xi_i>\rho\}\bigr).
\]
This proves \eqref{eq:pi} and \eqref{eq:coupled}. Independence follows because these events depend on independent pairs $(X_i,\xi_i)$. Also $G_i(x_1^*)=q_i$, so
\[
U_1(x_1^*)=\prod_{i=2}^n q_i,
\]
which proves \eqref{eq:constructivevalue} and implies that each $q_i>0$.

Finally, consider the rule for the $V$-sequence that rejects its first observation, then stops at the first upper value, with forced stopping at $n$ if no upper value appears. It wins exactly when there is one upper value among observations $2,\ldots,n$. Its success probability is
\[
\sum_{i=2}^n p_i\prod_{\substack{2\le j\le n\\j\ne i}}q_j
=\left(\prod_{i=2}^nq_i\right)\sum_{i=2}^n r_i.
\]
This cannot exceed the optimal value $\prod_{i=2}^nq_i$. Division by that positive product proves \eqref{eq:odds}.
\end{proof}

\begin{remark}[What is retained from the original reduction]
For continuous observations, \eqref{eq:pi} reduces to $p_i=\Pp(X_i>c)$, as in \citet{AI}. With atoms, the extension assigns the same fraction $1-\rho$ of the boundary mass to the upper value for every $i$. The construction uses auxiliary uniforms, but the comparison with the original sequence concerns ordinary stopping times, as proved by freezing those uniforms in Lemma~\ref{lem:spread}.
\end{remark}

\begin{corollary}[Known sharp bound]\label{cor:bound}
For independent real-valued observations and $n\ge2$,
\begin{equation}\label{eq:bound}
V_n(X_1,\ldots,X_n)\ge\left(1-\frac1n\right)^{n-1}.
\end{equation}
The bound is attained in both the discrete and continuous classes. For $n=1$, the value is $1$.
\end{corollary}
\begin{proof}
Apply Theorem~\ref{thm:main} and the bound for $V$-sequences in \citet[Section 3]{AI}, stated in Lemma~\ref{lem:publishedbound}. The examples below show that the bound is attained.
\end{proof}

\section{Attainment and the role of ties}
The extremal $V$-sequence in \citet[Section 3]{AI} has $X_1'\equiv0$, $a_i=i$, $b_i=-i$, and independent probabilities of the upper values $p_i=1/n$ for $i=2,\ldots,n$. Its optimal value, as established there, is
\[
V_n(X_1',\ldots,X_n')=\left(\frac{n-1}{n}\right)^{n-1}=:Q.
\]
This proves discrete attainment in the larger class considered here.

For continuous attainment, use the example of \citet{AI}: let $X_1$ be uniform on $(0,1)$ and, for $i=2,\ldots,n$, give $X_i$ density
\[
f_i(x)=\begin{cases}
1/n,&i<x<i+1,\\
(n-1)/n,&-i<x<-i+1,\\
0,&\text{otherwise}.
\end{cases}
\]
These intervals have the same ordering as the values of the preceding $V$-sequence. The position within an interval changes neither which observation is largest nor the probabilities of future upper and lower values. More explicitly, generate these positions independently of the upper and lower choices. Fixing the positions turns any stopping rule into a rule for the $V$-sequence. Averaging shows that its win probability cannot exceed $Q$. Conversely, a rule for the $V$-sequence can be followed by observing which interval contains each $X_i$. Hence the optimal win probability is $Q$.

\begin{remark}
The convention that a tie for the maximum is a win is essential. If success requires selecting a unique strict maximum, a sequence of identical constants has success probability zero. The present reduction concerns exactly the objective $\Pp(X_\tau=M_n)$.
\end{remark}

\end{document}